\documentclass[10pt]{article}

\usepackage{amsmath}
\usepackage{amssymb}
\usepackage{enumerate}
\usepackage{url}
\usepackage{amsfonts, amscd}
\usepackage{ epsfig, amsthm}

\usepackage{mdwlist}

\begin{document}

\newcommand{\be}{\begin{enumerate}}
\newcommand{\ee}{\end{enumerate}}
\newcommand{\MD}{\mathrel{\mathcal {D}}}
\newcommand{\CC}{{\mathcal C}}
\newcommand{\CR}{{\mathcal R}}
\newcommand{\CL}{{\mathcal L}}

\newcommand{\CH}{{\mathcal H}}
\newcommand{\CJ}{{\mathcal J}}

\newtheorem{theorem}{Theorem}[section]
\newtheorem{propo}[theorem]{Proposition}
\newtheorem{result}[theorem]{Result}
\newtheorem{lemma}[theorem]{Lemma}

\newtheorem{example}[theorem]{Example}
\newtheorem{corol}[theorem]{Corollary}

\title{A note on the Howson property in inverse semigroups}

\author{Peter R. Jones\\
Department of Mathematics, Statistics and Computer Science\\
Marquette University\\
Milwaukee, WI 53201, USA\\
peter.jones@mu.edu}
\date{\today}

\maketitle

\begin{abstract} An algebra has the Howson property if the intersection of any two finitely generated subalgebras is finitely generated. A simple necessary and sufficient condition is given for the Howson property to hold on an inverse semigroup with finitely many idempotents.  In addition, it is shown that any monogenic inverse semigroup has the Howson property.



\end{abstract}
\section{Introduction}\label{intro}  An algebra has the \textit{Howson property} if the intersection of any two finitely generated subalgebras is again finitely generated.  The eponym `Howson' stems from \cite{howson}, where it was shown that free groups have this property.  Motivated by recent work by P.V. Silva and F. Soares \cite{silva_soares}, we find a remarkably simple characterization of the inverse semigroups having the Howson property, under the assumption that the semilattice of idempotents is finite: the Howson property holds if and only if the same is true for its maximal subgroups.  When specialized to $E$-unitary semigroups, it can immediately be deduced that the Howson property holds if and only if the same is true for its maximal group image.  Thus we generalize the main theorem of \cite{silva_soares} using only elementary methods based on Green's relations.  Some of the technical ideas were motivated by techniques from the authors work \cite{jones_LF} on full inverse subsemigroups.

Throughout, inverse semigroups are to be regarded as unary semigroups.  Thus a group has the Howson property, regarded as a group, if and only if the same is true regarded as an inverse semigroup.  Perhaps the deepest earlier work on the Howson property for inverse semigroups was by P.G. Trotter and the author, who showed \cite{jones_trotter} that although free inverse semigroups of rank one have the Howson property, the property fails for free inverse semigroups of rank greater than one.  Silva also showed \cite{silva}, however, that the intersection of any two monogenic inverse subsemigroups of any free inverse semigroup is again finitely generated.

Before providing technical background, we connect this work with the cited paper \cite{silva_soares}, which concerns itself with the ($E$-unitary) inverse semigroups that are the semidirect products of semilattices and groups.  Their first main result (Theorem 3.4, that the resulting semigroup has the Howson property if and only if the group itself does) is proved under the hypothesis that the semilattice is finite, and thereby follows from our cited theorem (see Corollary~\ref{E-unitary}).  Although they do not make this observation, we may note that in fact our corollary may be decuced from their result, by virtue of O'Carroll's embedding theorem \cite{ocarroll}.  See the discussion following the cited corollary for more details.
\section{Preliminaries}\label{prelim} We refer the reader to the text by Petrich \cite{petrich} for inverse semigroups in general, including background on Green's relations, congruences and so on. The semilattice of idempotents of an inverse semigroup $S$ is denoted $E_S$. If $X \subseteq S$, then $\langle X \rangle$ denotes the inverse subsemigroup generated by $X$. An inverse semigroup is \textit{monogenic} if it is generated by a single element.  The notation $T \leq S$ means that $T$ is an inverse subsemigroup of $S$.

The \textit{Brandt} semigroups \cite[Section II.3]{petrich} are the completely 0-simple inverse semigroups.  They are the semigroups isomorphic to the following: given a group $G$ and a nonempty set $I$,  $B(G,I)$ consists of the set $  I \times G \times I \cup \{0\}$, where $(i , g, j)(j, h, k) = (i, gh, k)$ and all other products are 0. The nonzero idempotents are therefore in one-one correspondence with the set $I$.  Every primitive inverse semigroup --- one in which every nonzero idempotent is minimal --- is the 0-direct union of Brandt semigroups.

Let $S$ be an inverse semigroup.  For each $\CJ$-class $J$ of $S$, let $PF(J)$ denote the \textit{principal factor} associated with $J$. Formally (\cite[Section I.6]{petrich}), if $J = J_a$ is not the least $\CJ$-class of $S$, $PF(J) = J(a)/I(a)$, where $J(a) = SaS$ and $I(a) = J(a) \backslash J_a$.  In practice, we consider $PF(J)$ as $J \cup \{0\}$, with the products from $J$ being those in $S$, if they lie in $J$, all other products being 0.  If $J$ is the least $\CJ$-class of $S$, then $PF(J)$ is defined to be $J$ itself.

If $J$ is the least $\CJ$-class of $S$ (its \textit{kernel}), then $PF(J)$ is a group.  Otherwise, $PF(J)$ is 0-simple and, if $E_J$ is finite, completely 0-simple and thus a Brandt semigroup.  In general, $S$ is \textit{completely semisimple} if each principal factor is a group or a Brandt semigroup.  This holds if and only if $S$ contains no bicyclic subsemigroup.  In such a semigroup, ${\cal D} = \CJ$.

We shall represent the bicyclic semigroup $B$ via the presentation $\langle x : xx^{-1}\geq x^{-1}x\rangle$, in which case $E_S $ is the chain $e_0 > e_1 > e_2 > \cdots$, where $e_i = x^{-i}x^i$ ($x^0$ corresponding to the identity element $1 = x x^{-1}$).  For alternative representations, see \cite{petrich}.

An inverse semigroup $S$ is $E$-\textit{unitary} if whenever $s \in S, e \in E_S$ and $es \in E_S$ then $s \in E_S$.  Let $\sigma = \{(s,t) \in S \times S : es = et \text{ for some } e \in E_S\}$ denote the least group congruence on any inverse semigroup, so that $S/\sigma$ is its maximal group homomorphic image.  Then \cite[Proposition III.7.2]{petrich} $S$ is $E$-unitary if and only if $E_S$ is a $\sigma$-class of $S$ and if and only if $\CR \cap \sigma$ is the identity relation.

Note that if the $E$-unitary inverse semigroup $S$ has a least idempotent $e$, say, then its kernel $J_e = H_e$ is isomorphic to $S/\sigma$.


\section{The main theorem}\label{howson}

\begin{propo}\label{brandt fg} A Brandt semigroup $B( G, I)$ is finitely generated if and only if $I$ is finite and $G$ is finitely generated.
\end{propo}

\begin{proof}  First, suppose $I$ is finite,  $I = \{1, 2, \ldots, n\}$, say, and let $A$ be any generating set for $G$.  As usual, the nonzero $\CR$-classes and the nonzero $\CL$-classes may be indexed by $I$ and thus the nonzero $\CH$-classes by $I \times I$, in such a way that the group $\CH$-classes are $\{H_{ii}: i \in I\}$, all isomorphic with $G$.  Let $x_1, x_2, \ldots, x_n$ be a transversal of the $\CH$-classes $H_{11}, H_{12}, \ldots, H_{1n}$, with $x_1 $ the identity element $e$ of $H_{11}$.  By Green's lemma \cite[Lemma I.6.9]{petrich}, $H_{ij} = x_i^{-1} H_{11} x_j$ for each $(i,j) \in I \times I$, so $B(G, I) = \langle A \cup \{ x_2, \ldots, x_n\} \rangle$.  Clearly, if $G$ is finitely generated, then so is $B(G,I)$.  

Conversely, if $B(G, I)$ is finitely generated, then $I$ is finite (since any nonzero element is $\CR$-related to either a generator or the inverse of a generator).  In the notation of the previous paragraph, $B(G,I) = \langle B \cup \{ x_2, \ldots, x_n\} \rangle$ for some finite subset $B$ of $A$.  Now if $g \in H_{11}$, then the only nonzero products with members of $\{ x_2, \ldots, x_n\} \cup \{ x_2^{-1}, \ldots, x_n^{-1}\}$ are of the form $gx_j$, $x_i^{-1}g$ and $x_i^{-1}gx_j$.  But the only nonzero products of members of $\{ x_2, \ldots, x_n\} \cup \{ x_2^{-1}, \ldots, x_n^{-1}\}$ are $x_j x_j^{-1} = e$ and $x_j^{-1}x_j \in H_{jj}$.  Hence $H_{11} = \langle B \rangle$ and $G$ is finitely generated.
\end{proof}

\begin{corol}\label{brandt howson} A Brandt semigroup $B( G, I)$ has the Howson property if and only if the same is true of $G$.
\end{corol}

\begin{proof} Necessity is clear.  To prove the converse, first consider any inverse subsemigroup $T$ of $S = B(G,I)$ that is not just a subgroup.  Then $T$ is primitive and thus a 0-direct union of Brandt subsemigroups.  Clearly, $T$ is finitely generated if and only if there are finitely many factors, each of which is finitely generated. So by Proposition~\ref{brandt fg} it is finitely generated if and only if each factor has finitely many idempotents and its maximal subgroups (which are subgroups of $G$) are finite generated.

Then let $U$ and $V$ be finitely generated inverse subsemigroups of $S$ that are not just subgroups.  If $U \cap V$ is nonempty and not $\{0\}$, first suppose it is a nonzero group, a subgroup of $H_e$, say, for some $e \in E_S$ .  Then it is a common subgroup of a 0-direct factor of $U$ and a 0-direct factor of $V$ and therefore it is the intersection of the subgroups $U \cap H_e$ and $V \cap H_e$.  By the preceding paragraph, each of these subgroups is finitely generated and therefore, by the Howson property, so is $U \cap V$.

If $U \cap V$ is not just a subgroup, then, similarly, each factor in its 0-direct union is contained within a 0-direct factor of $U$ and a 0-direct factor of $V$ and so is the intersection of those two factors.  Each of its maximal subgroups is therefore the intersection of a maximal subgroup of $U$ with a maximal subgroup of $V$ and is therefore finitely generated.  By Proposition~\ref{brandt fg}, $U \cap V$ is finitely generated.
\end{proof}

\begin{lemma}\label{tech} Let $S$ be an inverse semigroup, $T = \langle X \rangle$ an inverse subsemigroup of $S$, and $J$ a $\CJ$-class of $S$.  Then $T \cap J \subseteq \langle E_J X \cap T \cap J \rangle$.
\end{lemma}

\begin{proof}  Let $t \in T \cap J$, $t = y_1 \cdots y_n$, where for each $i$, $y_i$ or $y_i^{-1}$ belongs to $X$.  Then by the result of Hall \cite[Lemma 1]{hall}, $t = \bar{y_1} \cdots \bar{y_n}$, where for each $i$, $\bar{y_i} = e_i y_i$, $e_i = y_i \cdots y_n  t^{-1} y_1 \cdots y_{i-1}$ (with the usual provisions in case $i = 1$ or $i = n$).  By the cited result, each $\bar{y_i} \MD t$ and each $e_i \in E_J$. Clearly each $e_i \in T$.  Note that if $y_i = x^{-1}$ then $e_i y_i = (f_i x)^{-1}$, where $f_i = x e_i x^{-1} \in E_J \cap T$.  Thus $t \in \langle E_J X \cap T \cap J \rangle$.
\end{proof}

Let $T \leq S$ and $J = J_a$ a non-group $\CJ$-class of $S$ that $T$ meets nontrivially.  Then $T \cap I(a) \leq I(a)$ and so $(T \cap J) \cup \{0\}$ may be regarded as an inverse subsemigroup of $PF(J)$.  Denote it by $T_J$.  If $J$ is a subgroup, put $T_J = T \cap J$.

\begin{propo}\label{T_J} Let $S$ be an inverse semigroup with finitely many idempotents and $T \leq S$.  Then $T$ is finitely generated if and only if $T_J$ is finitely generated for each $\CJ$-class $J$ of $S$ that meets $T$ nontrivially.
\end{propo}

\begin{proof}  If $T $ is generated by the finite set $X$ then, by Lemma~\ref{tech}, each such $T_J$ is generated by the finite set $ E_J X \cap T \cap J$.  Conversely, if for each such $J$, $T_J = \langle X_J \rangle$, where we may assume $X_J$ does not contain the zero of $PF(J)$ in the non-group case, then when interpreted in $S$ each $X_J$ is contained in $T$ and $T \cap J \subseteq \langle X_J \rangle$.  Thus $T$ is generated by the union of the sets $X_J$.
\end{proof}

We can now prove the theorem stated in the introduction.

\begin{theorem}\label{theorem} Let $S$ be an inverse semigroup with finitely many idempotents.  Then $S$ has the Howson property if and only if each of its maximal subgroups has this property.
\end{theorem}

\begin{proof} Necessity is clear.  For the converse, let $U$ and $V$ be finitely generated inverse subsemigroups of $S$.  Let $J$ be a $\CJ$-class of $S$ that meets $U \cap V$ nontrivially.  

If $J$ is the minimum $\CJ$-class of $S$, then all the relevant intersections are subgroups of $J$ and so $(U \cap V)_J$ is finitely generated, by the Howson property for $J$.

Otherwise, by the direct half of Proposition~\ref{T_J}, $U_J$ and $V_J$ are finitely generated inverse subsemigroups of $PF(J)$.  Applying Proposition~\ref{brandt howson}, $U_J \cap V_J$ is finitely generated.  But $U_J \cap V_J = ((U \cap J) \cap (V \cap J)) \cup \{0\} = ((U \cap V) \cap J) \cup \{0\} = (U \cap V)_J$.  Applying the reverse half of Proposition~\ref{T_J}, $U \cap V$ is therefore finitely generated and the Howson property holds in $S$.
\end{proof}

Finiteness of the semilattice of idempotents is necessary: as remarked in the introduction, the free inverse semigroups of rank greater than one do not satisfy the Howson property.  But all their subgroups are trivial.

\begin{corol}\label{E-unitary} Let $S$ be an $E$-unitary inverse semigroup with finitely many idempotents.  Then $S$ has the Howson property if and only if the maximal group image $S/\sigma$ has this property.  In particular \cite[Theorem 3.4]{silva_soares}, the semidirect product of a finite semilattice and a group has the Howson property if and only if the group has this property.
\end{corol}

\begin{proof}  In this special case, the minimum $J$-class $J$ is isomorphic to $S/\sigma$ and each maximal subgroup of $S$ embeds in $J$, so the result is immediate from Theorem~\ref{theorem}.
\end{proof}

Remark: in \cite[Corollary 4.1]{silva_soares}, Silva and Soares cleverly apply the last statement of the corollary just stated to extend it to the case of locally finite actions: those for which the orbit of each element of the semilattice under finitely generated subgroups is finite.  One application was to actions on `strongly finite above semilattices with identity'.  These are semilattices possessing a well-defined height function, given by the maximum length of a chain from each element to the identity, and, further, having only finitely many elements of any given height.  Note that the semilattice of idempotents of any free inverse semigroup of finite rank has this property.  Thus it does not guarantee the Howson property for $E$-unitary inverse semigroups in general.

Another application of the same corollary --- to the case that the group is \textit{locally finite} --- is moot: any $E$-unitary inverse semigroup whose maximal group homomorphic image is locally finite is itself locally finite, by Brown's lemma \cite{brown}, and therefore immediately has the Howson property.

On the other hand, as remarked in the introduction, Corollary~\ref{E-unitary} may also be deduced from \cite[Theorem 3.4]{silva_soares}.  Sufficiency follows from O'Carroll's theorem \cite{ocarroll} that any $E$-unitary inverse semigroup $S$ [with finite semilattice of idempotents and] maximum group homomorphic image $G$ embeds in a semidirect product of a [finite] semilattice with $G$.  Necessity follows from the fact that $S/\sigma$ is isomorphic to the group kernel of $S$.

\section{Monogenic inverse semigroups}\label{B} As remarked in the introduction, the author and P.G. Trotter showed that the free inverse semigroup $FI_1$ of rank one (which is an $E$-unitary inverse semigroup) has the Howson property.  A key tool was the following.

    \begin{result}\cite[from Proposition 1.5]{jones_trotter} For any inverse subsemigroup $S$ of the free inverse semigroup of rank one, the inverse subsemigroup generated by the nonidempotents of $S$ is finitely generated.
  \end{result}  
  
   As far as the author is aware, all the inverse semigroups hitherto known to have the Howson property were completely semisimple, so the following results are of special interest.
    
  \begin{lemma}\label{fg}  Every inverse subsemigroup of the bicyclic semigroup $B$ that does not consist solely of idempotents is finitely generated. 
  \end{lemma}
  
  \begin{proof} Let $S$ be such a subsemigroup and let $e_k$ be the greatest idempotent for which the associated $\CR$-class of $S$ is nontrivial, where $k$ a nonnegative integer.  Say $e_k x^m \in S$, where $m$ is a positive integer.  Then for every nonnegative integer $r$, $S$ contains $(e_k x^m)^r = e_k x^{mr}$ and so $(e_k x^m)^{-r} (e_k x^m)^r = e_{k + rm}$.
  
Since $B$ is a quotient of the free inverse semigroup of rank one, $S$ is generated by a finite set $A$ of nonidempotents, together with a set $F$ of idempotents.   Suppose $S$ contains $e_{k+sm+i}$ for some $s$ and some $i$, $0 < i < m$.  Since $x^{sm}x^{-sm} = 1$, direct calculation shows that $S$ contains $(e_k x^{sm}) e_{k+sm + i} (e_k x^{sm})^{-1} = e_{k+i}$ and $e_{k+i} (e_k x^m) = e_{k+i} x^m$, in which case, further, $S$ contains $e_{k + rm +i}$ for every nonnegative $r$, as in the previous paragraph.
  
 Thus $S$ is generated by $A$ together with the finite set of idempotents $\{e_j : 0 \leq j < k + m\} \cap S$.
  \end{proof}

  \begin{corol}\label{bicyclic} The bicyclic semigroup $B$ has the Howson property.
  \end{corol}
  
  \begin{proof} Let $U, V$ be finitely generated inverse subsemigroups.  If either consists solely of idempotents, then it is finite, so assume otherwise.  Then for some positive integers $m$ and $n$ and nonnegative integers $j$ and $k$, $u = e_j x^m \in U$ and $v = e_k x^n \in V$.  By replacing $u$ and $v$ respectively by $u^n$ and $v^m$, if necessary, we may assume $m = n$, so that $u \mathrel{\sigma} v$. 
  
If $U \cap V$ contains a nonidempotent, then Lemma~\ref{fg} applies.  In particular, this is the case whenever $U \cap V$ contains an idempotent $ e_i$ such that $i \geq \max(j,k)$, that is, $e_i \leq uu^{-1}, vv^{-1}$: for then $e_iu \mathrel{\CR} e_i v$ and $e_i u \mathrel{\sigma} e_i v$, so since $B$ is $E$-unitary, $e_i u = e_i v $ is a nonidempotent of $U \cap V$.

 Otherwise, $U \cap V$ is contained in the finite set $\{e_i : i < \max (j,k)\}$.
  \end{proof}
  
 Apart from the free one, the monogenic inverse semigroups are determined by the following classes of defining relations \cite[Theorem IX.3.11]{petrich}, where $k, \ell$ are positive integers: (1) $x^k = x^{k + \ell}$; (2) $x^k x^{-1} = x^{-1} x^k$; (3) $x^k = x^{-1} x^{k+1}$.

The first class consists of the finite instances.  The second consists of the the infinite cyclic group ($k = 1$) and ideal extensions of that group by finite Rees quotients of $FI_1$.  The third consists of the bicyclic semigroup $B$  ($k = 1$) and ideal extensions of $B$ by finite Rees quotients of $FI_1$.
  
\begin{propo}\label{monogenic}  Every monogenic inverse semigroup has the Howson property.
\end{propo}
\begin{proof} We have already noted that $FI_1$ has this property.  For the first class this is obvious; those in the second class have finitely many idempotents and the maximal subgroups are either trivial or infinite cyclic, so Theorem~\ref{theorem} applies.  We have just proved that $B$ has the Howson property, so let $S$ be an ideal extension of $B$ by a finite Rees quotient of $FI_1$, where $x^n \in B$, say.  Let $U, V$ be finitely generated inverse subsemigroups of $S$.  As in the proof of Corollary~\ref{bicyclic}, we may assume that neither $U$ nor $V$ consists of idempotents. Since $x$ is of infinite order, neither $U \cap B$ nor $V \cap B$ consists of idempotents.  By Lemma~\ref{fg}, $U \cap B$ and $V \cap B$ finitely generated, so $(U \cap V) \cap B$ is finitely generated and, since $(U \cap V) \backslash B$ is finite, $U \cap V$ is also finitely generated.
\end{proof}


\begin{thebibliography}{99}

\bibitem{brown}
	T.C. Brown, An interesting combinatorial method in the theory of locally finite semigroups, Pacific J. Math. 36 (1971), 285--289.
	
	
	
\bibitem{hall}
	T.E. Hall, On regular semigroups, J. Algebra 24 (1973), 1--24.
	
\bibitem{howson}
	A.G. Howson, On the intersection of finitely generated free groups, J. London Math. Soc. 29 (1954), 428--434.
	
\bibitem{jones_LF}
	P.R. Jones, Semimodular inverse semigroups, J. London Math. Soc. (2) 17 (1978), 446--456.
	
\bibitem{jones_trotter}
	P.R. Jones and P.G. Trotter, The Howson property for free inverse semigroups, Simon Stevin 63 (1989), 277--284.
	
\bibitem{ocarroll}
	L. O'Carroll, Embedding theorems for proper inverse semigroup, J. Algebra 42 (1976), 26 -- 40.

\bibitem{petrich}
	M. Petrich, Inverse Semigroups, Wiley, New York, 1984.
	
\bibitem{silva}
	P.V. Silva, `Contributions to combinatorial semigroup theory' (Ph.D. Thesis), University of Glasgow, 1991.
	
\bibitem{silva_soares}
	P.V. Silva and F. Soares, Howson's property for semidirect products of semilattices by groups, arXiv:1412.3048v1.
\end{thebibliography}
\end{document}